\documentclass[12pt]{article}

\usepackage[T1]{fontenc}
\usepackage[latin9]{inputenc}
\usepackage{geometry}
\usepackage{refstyle}
\usepackage{amsthm}
\usepackage{amsmath}
\usepackage{amssymb}
\PassOptionsToPackage{normalem}{ulem}
\usepackage{ulem}

\makeatletter

\AtBeginDocument{}
\AtBeginDocument{\providecommand\eqref[1]{\ref{eq:#1}}}
\AtBeginDocument{\providecommand\lemref[1]{\ref{lem:#1}}}
\RS@ifundefined{subref}
  {\def\RSsubtxt{section~}\newref{sub}{name = \RSsubtxt}}
  {}
\RS@ifundefined{thmref}
  {\def\RSthmtxt{theorem~}\newref{thm}{name = \RSthmtxt}}
  {}
\RS@ifundefined{lemref}
  {\def\RSlemtxt{lemma~}\newref{lem}{name = \RSlemtxt}}
  {}

\usepackage[all]{xy}

\numberwithin{equation}{section}
\def\id{\text{id}}
\def\Id{\operatorname{Id}}

\def\End{\operatorname {End}}

\def\End{\operatorname {End}}

\DeclareMathOperator{\Span}{span}
\DeclareMathOperator{\sign}{sign}

\DeclareMathOperator{\Par}{Par}

\global\long\def\MM{M_{\mathbf{tot}}}
\global\long\def\MMM{M_{\mathbf{total}}}

\newtheorem{lemma}{Lemma}[section]

\newtheorem{theorem}[lemma]{Theorem}
\newtheorem{corollary}[lemma]{Corollary}

\newtheorem{convention}[lemma]{Convention}

\theoremstyle{definition}

\newtheorem{definition}[lemma]{Definition}
\newtheorem{conjecture}[lemma]{Conjecture}

\theoremstyle{remark}

\newtheorem{remark}[lemma]{Remark}

\newtheorem{notation}{Notation}

\def\uline{\underline}

\mathchardef\mhyphen="2D

\title{The polynomial part of the codimension growth of affine PI algebras}
\author{}
\begin{document}

\author{Eli Aljadeff, ~Geoffrey Janssens and Yakov Karasik
\thanks {The first and third authors were supported by the ISRAEL SCIENCE FOUNDATION (grant No. 1017/12). The second author was supported by the FWO Ph.D fellowship}}

\maketitle
\begin{abstract}
Let $F$ be a field of characteristic zero and $W$ an associative affine $F$-algebra satisfying a polynomial identity (PI). The codimension sequence \{$c_n(W)$\} associated to $W$ is known to be of the form $\Theta (n^t d^n)$, where $d$ is the well known PI-exponent of $W$. In this paper we establish an algebraic interpretation of the polynomial part (the constant $t$) by means of Kemer's theory. In particular, we show that in case $W$ is a \textit{basic} algebra (hence finite dimensional), $t = \frac{q-d}{2} + s$, where $q$ is the number of  simple component in $W/J(W)$ and $s+1$ is the nilpotency degree of $J(W)$ (the Jacobson radical of $W$). Thus proving a conjecture of Giambruno.

%We study some numerical invariants associated to affine algebras $W$ that satisfy a polynomial identity. The object of the investigation is the so called codimension sequence $c_n(W)$ which measures the number of multilinear non-identities on $W$. In case $W$ is unital, $c_n(W)$ asymptotically behaves as $f(n) = c n^t d^n$. By a result of Giambruno and Zaicev the exponential growth of this sequence is well understood and proven to be tightly connected with the algebraic structure of $W$. However, the polynomial part, i.e. the power $t$, of $c_n(W)$ was still not understood. A conjecture of Giambruno gives a description of $t$ in function of algebraic data of the basic components of $W$. In this article we prove this conjecture.
\end{abstract}

\section{Introduction}
Let $W$ be an affine (i.e. finitely generated) PI algebra over a field $F$ of characteristic $0$. The codimension growth of $W$ was studied heavily in the last 40 years or so (see for instance \cite{GiaZai, BelKarRow}). It provides an important tool for measuring the ''size'' of the $T$-ideal of identities of $W$ in asymptotic terms. In particular, the exponential part of the asymptotics is key in the classification of varieties of PI algebras. Let us recall briefly some definitions.

Let $F \langle X \rangle$ be the free algebra over a countable set $X$. A polynomial $f(x_1, \ldots, x_n) \in  F \langle X \rangle$ in noncommuting indeterminates $x_1, \ldots, x_n \in X$, is called a polynomial identity of $W$ if $f(w_1, \ldots, w_n)=0$ for any $(w_1,\ldots,w_n) \in W^{n}$. We say $W$ is a PI-algebra if such nonzero polynomial exists. The set of polynomial identities of $W$, denoted $\Id(W)$, is clearly an ideal of $F \langle X \rangle$. Moreover $\Id(W)$ is a $T$-ideal, that is, it is invariant under any endomorphism $\phi \in \End_F(F \langle X \rangle)$. It is well known (applying multilinearization) that $T$-ideals are generated as $T$-ideal by multilinear elements, i.e by the elements in $P_n(F) =  \Span_F \{ x_{\sigma(1)} \cdots x_{\sigma(n)} \mid \sigma \in S_n \}$, $n \in \mathbb{N}$. Let $c_n(W) = \dim_F P_n(F)/ P_n(F) \cap \Id(W)$ be the $n^{th}$ term of the codimension sequence of $W$. It is clear that $c_n(W) = c_n(W \otimes_F R)$ where $R$ is a commutative domain and hence, for the calculation of $c_n(W)$ we may assume that {\it $F$ is an algebraically closed field}.

Regev, in his pioneering article, proved the sequence $\{c_n(W)\}$ is exponentially bounded \cite{RegTensor}. Twenty five years later Giambruno and Zaicev showed (as conjectured by Amitsur) that $ \lim\limits_{n \rightarrow \infty} \sqrt[n]{c_n(W)}$ exists and is an integer. This is the PI-exponent of $W$, denoted by $\exp(W)$ \cite{GiaZai98}.

Giambruno and Zaicev proved that $\exp(W)$ is determined by the $F$-dimension of a suitable subalgebra of $W$ in case $W\cong A$ is finite dimensional over $F$. Let us state their result precisely. Recall that by Wedderburn-Malcev's theorem any finite dimensional algebra $A$ over $F$ can be decomposed into $A = A_{ss} \oplus J(A)$, where $A_{ss}$ is a maximal semisimple subalgebra of $A$ (unique up to isomorphism) and $J(A)$ is the Jacobson radical of $A$. Moreover, $F$ being algebraically closed, $A_{ss} \cong  A_{1} \times \cdots \times A_{q}$, where the product is direct and $A_{i} \cong  M_{d_i}(F)$ is the algebra of $d_i \times d_i$-matrices over $F$.
With this notation, Giambruno and Zaicev proved that

$\exp(A) = \max \{dim_{F}(A_{i_1}\times \cdots \times A_{i_r}): $ where $ i_{j} \neq i_{k}$ for $j \neq k$ and $ A_{i_1}J A_{i_2} \cdots J A_{i_r} \neq 0 \}.$

In order to get their result for arbitrary affine algebras (namely $ \lim\limits_{n \rightarrow \infty} \sqrt[n]{c_n(W)}$ exists and is an integer) one needs to invoke the fundamental representability theorem of Kemer \cite{Kemer,AljBelKar}.

\begin{theorem}\label{representability theorem}
For any affine PI algebra $W$ there exists a finite dimensional algebra $A$ such that $\Id(W) = \Id(A)$.
\end{theorem}
Regev conjectured that the asymptotic behavior of the sequence $\{c_n(W)\}$ is of the form
$c n^t d^n$ for some constants $c \in \mathbb{R}, t \in \frac{1}{2}\mathbb{Z}$ and $d \in \mathbb{Z}_{\geq 0}$ (we write $c_n(W) \simeq c n^t d^n$ where $f \simeq g$ means $\lim \frac{f}{g} =1$). A weakened version of Regev's conjecture was confirmed by Berele and Regev in \cite{BerRegAsymp}, namely that $\{c_n(W)\}$ is asymptotically bounded by the functions
\begin{equation}\label{eq: regev conj}
c_1 n^{t} (\exp(W))^n \lesssim c_n(W) \lesssim c_2 n^{t} (\exp(W))^n
\end{equation}
for some constants $c_1, c_2 \in \mathbb{R}$ and  $t \in \frac{1}{2}\mathbb{Z}$ in case the codimension sequence is eventually nondecreasing (i.e. monotonic nondecreasing for large enough $n$). Furthermore, they showed that if $W$ is unital (and hence the codimension sequence is nondecreasing) $c_1 = c_2$ and thus confirming in this case Regev's conjecture. Recently, in \cite{GiaZaiEve}, Giambruno and Zaicev proved that the sequence of codimensions is indeed eventually nondecreasing, showing the asymptotic inequality above holds for any affine PI algebra.
We refer to $t=t(W)$ as the polynomial part of the codimension sequence of $W$.

We emphasize that the proof of Berele and Regev for the existence of the parameter $t(W)$ does not give a formula for its calculation. In this article (following a reduction to basic algebras) we present an interpretation, a la Giambruno and Zaicev, of the polynomial part of the codimension growth for any finite dimensional $F$-algebra $W$.

Let $A$ be a finite dimensional algebra and let $\mbox{Par}(A) = (\dim_{F}A_{ss}, \mbox{nildeg}(J(A)) -1)$ be its \textit{parameter} (nildeg$(J(A))$ denotes the nilpotency degree of the Jacobson radical). Such an algebra $A$ is said to be basic (or fundamental) if it is not PI equivalent to $B_1 \times \cdots \times B_l$ where $\mbox{Par}(B_i) < \mbox{Par}(A)$ for every $1 \leq i \leq l$. By induction one can easily get the following

\begin{lemma}\label{Basic decomposition}
Let $A$ be a finite dimensional algebra over $F$. Then there exist basic algebras $B_1, \ldots, B_l$ such that $A$ is PI-equivalent to $B_1 \times \cdots \times B_l$.
\end{lemma}

This can be used to reduce the problem of interpreting the polynomial part (from arbitrary finite dimensional algebras) to basic algebras.

\begin{corollary} \label{reduction to basic algebras}

With the above notation,
$\exp(A)=\max_{1 \leq i \leq l}\exp(B_{i})$ and

$$ t(A) = \max_{j}\{t(B_j): \exp(B_{j})=\exp(A) \}.$$

\end{corollary}

\begin{remark}
The definition of a \textit{basic algebra} used in the proof of the representability theorem for affine PI algebras (Theorem \ref{representability theorem}) is different, yet equivalent,  to the one we presented above. The decomposition of finite dimensional algebras into the direct product of basic algebras up to PI equivalence (Lemma \ref{Basic decomposition} above) using the other definition is a key and nontrivial step in the proof of the Theorem.
 \end{remark}

For basic algebras Giambruno made the following conjecture.

\begin{conjecture}Let $A$ be a basic algebra with Wedderburn-Malcev decomposition $A \cong M_{d_1}(F) \times \cdots \times M_{d_q}(F) \oplus J(A)$.  Then
$$t(A) = \frac{q-d}{2} + s$$
where $d=d_1^{2}+ \cdots +d_{q}^{2}$ and $s+1$ is the nilpotency degree of $J(A)$.
\end{conjecture}

In case $A=M_d(F)$ this was established by Regev \cite{RegAsymp}. The conjecture is also known to hold for the algebra of upper-block triangular matrices $UT(d_1, \ldots, d_q)$  \cite{GiaZai03}. This was proved by Giambruno and Zaicev. In their proof they used Lewin's theorem \cite{Lewin} and Berele and Regev's result \cite{BerReg}.

Applying Regev's result for matrix algebras, Giambruno's conjecture can be re-stated as follows. Let $A$ be a basic algebra over $F$ and
$A \cong A_{1}\times \cdots A_{q} \oplus J(A)$
be its Wedderburn-Malcev decomposition. Then

$$ t(A) = t(A_1) + \cdots + t(A_q) + (\mbox{nildeg}(J(A))-1).$$

Our main result appears is Theorem \ref{main theorem} where we prove the conjecture for an arbitrary basic algebra.

The paper is organized as follows. In section \ref{sec: basic algebras} we study basic algebras. In particular, to any basic algebra $A$ we associate another basic algebra $\mathcal{A}$ which is better understood and yet $\Par(\mathcal{A})=\Par(A)$. The associated algebra $\mathcal{A}$ will be the main tool in the proof of the upper bound, namely $t(A) \leq \frac{q-d}{2} + s$ (subsection \ref{sec: upper bound}). Finally in subsection \ref{sec: lower bound} we handle the lower bound and achieve the main result.

%The paper is organized as following. In section \ref{sec: prelim on pi} we recall the basic definitions and results on asymptotic PI theory. Then we define and study basic algebras in section \ref{sec: basic algebras}. In particular we explain that one can, in an equivalent way, describe basic algebras by means of its wedderburn-Malcev decomposition (algebraic data) or by means of combinatorial data, namely some extremal nonidentities. We also give important examples of basic algebras and define a basic algebra that can be associated to any finite dimensional algebra.
%
%Finally in section \ref{sec: proof Giam Conj} we proceed to the proof of the conjecture. As usually the equality is proved by distinguishing the upper bound and the lower bound.

\begin{notation}
Throughout, $F$ will be an alegebraically closed field of characteristic $0$, $W$ an affine algebra over $F$, $A$ a finite dimensional algebra over $F$, $A_{ss}$ a maximal semisimple subalgebra and $J(A)$ its radical. By $\mbox{nildeg}(J(A))$ we denote the nilpotency degree of $J(A)$, i.e. the number such that $J(A)^l=0$ but $J(A)^{l-1} \neq 0$. All polynomials are elements of $F\langle X \rangle$, the polynomial ring on countably many noncommuting variables (the elements of $X$). We denote by $S_n$ the symmetric group on $n$ letters and by $\mathbb{N}$ the set of positive integers.
\end{notation}

\section{Basic algebras} \label{sec: basic algebras}

\subsection{Preliminaries}

In this section we present a condition which is equivalent to an algebra being basic. In particular, basic algebras may be viewed as minimal models for a given Kemer index.

%{\it We always assume that the algebras are unital and defined over an algebraically closed field $F$ of characteristic $0$. The letter $A$ will always denote a finite dimensional algebra.}

%Let $W$ be an affine PI algebra. By Kemer's Theorem there exists a finite dimensional algebra $A$ such that $\Id(W) = \Id(A)$, in particular $c_n(W) = c_n(A)$. Since we are interested in the PI theoretical structure of $W$, we may work up to PI equivalence and thus assume that $W$ is actually finite dimensional (i.e. work with $A$ instead).
%
%It was proved by Berele and Regev that $c_1 n^t d^n \leq c_n(A) \leq c_2 n^t d^n$ for some constants $c_1,c_2 \in \mathbb{R}, t \in \frac{\mathbb{Z}}{2}$ and $d \in \mathbb{Z}$. The number $d$ is called the exponent of $A$ and $t$ the polynomial part of $A$. The exponent has been heavily studied and an explicit formula is known.
%
%Recall that by the Wedderburn-Malcev decomposition one can write $A= R \oplus J(A)$ where $R= R_1 \times \ldots \times R_q$ is a maximal semisimple subalgebra of $A$ and $J(A)$ is the Jacobson radical of $A$ (which is nilpotent since $A$ is finite dimensional). Suppose $R_i \cong M_{d_i}(F)$. Then
%$$\exp(A) = \max \{ \dim( R_{i_1} \oplus \ldots \oplus R_{i_k}) \mid R_{i_1} J(A) \cdots J(A) R_{i_k} \neq 0\}.$$

{\it As always in this article $A \cong A_1 \times \cdots \times A_q   \oplus J(A)$ with $A_i \cong M_{d_i}(F)$. Thus $d^{2}_i = \dim_{F} A_i$. }

\begin{definition}[Kemer index]
For any $\nu \in \mathbb{N}$ let
$$\Delta_{\nu} = \{ r \in \mathbb{N} \cup \{ 0 \} \mid \exists p(X) \notin \Id(A) \mbox{ alternating in $\nu$ disjoint sets of size r} \}.$$ Clearly if $\nu \leq \gamma$, then $\max \Delta_{\nu} \geq \max \Delta_{\gamma}$. Let $d = \lim_{\nu} \max (\Delta_{\nu})$.

Next we let
\begin{align*}
& S_{\nu}^d= \{ j \in \mathbb{N} \cup \{ 0 \} \mid \exists p(X) \notin \Id(A) \mbox{ alternating in $\nu$ disjoint sets of size d} \\
& \mbox{ and alternating in $j$ sets of size $d+1$}\}.
\end{align*}
Also here $\max S_{\nu}^d \geq \max S_{\gamma}^d $ if $\nu \leq \gamma$ and we set
$s = \lim_{\nu} \max S_{\nu}^d.$

The tuple $\kappa_A = (d,s)$ is called the {\it Kemer index of $A$.}
\end{definition}

Note that $d$ and $s$ are finite. Indeed, $d$ is finite because $\max \Delta_{\nu} \leq \dim A$ for any $\nu > 0$ and $s$ is finite due to the definition of $d$. The alternating sets of size $d$ are called {\it small sets} and the sets of size $d+1$ are called {\it big sets}.
In other words, there exist nonidentities in an arbitrary number of alternating small sets, but only a finite number of big sets (actually at most $s$ big sets).

\begin{definition}(Kemer polynomials)
Let $\nu_0$ be a number where $\max \Delta_{\nu_0}=d$ and $\max S_{\nu_0}^{d}=s$. Then a multilinear polynomial $f$ is called a {\it Kemer polynomial} of $A$ if $f \notin \Id(A)$ has at least $\nu_0$ small sets (cardinality $d$) and exactly $s$ big sets (cardinality $d+1$).
\end{definition}

For a general finite dimensional algebra one has $\kappa_A =(d,s) \leq \Par(A)$ in the lexicographic order \cite{AljBelKar}. Kemer showed that equality characterizes basic algebras.

\begin{theorem}(Kemer, see \cite{Kemer} or (\cite{AljBelKar}, Prop. 7.13)) \label{th: basis iff par= index}
A finite dimensional algebra $A$ is basic if and only if $\kappa_A = \Par(A)$.
\end{theorem}

An important consequence of this theorem is that Kemer polynomials do not vanish only for evaluations where all simple components $A_i$ are represented among the substitutions and precisely $J(A)-1$ variables are substituted by radical elements.

\subsection{Examples}
\subsubsection{Well known examples}
The matrix algebra $M_d(F)$ has Kemer index $(d^2,0)$ and thus is basic. To see this, note that any multilinear polynomial alternating on $d^2+1$ variables vanishes on $M_d(F)$. Now recall the $n$th Capelli polynomial

$$\mbox{cap}_n(X; Y) = \sum\limits_{\sigma \in S_n} \sign (\sigma) y_1 x_{\sigma(1)} \cdots y_{n} x_{\sigma(n)} y_{n+1}.$$
It is well known, \cite[Prop. 1.7.1]{GiaZai}, that $\mbox{cap}_{d^2+1}(X; Y) \in \Id (M_d(F))$ but $\mbox{cap}_{d^2}(X; Y) \notin \Id (M_d(F))$ and moreover all diagonal elementary matrices can be realized as a nonzero evaluation. Therefore for any $\mu \in \mathbb{N}$ the polynomial
$$\mbox{Cap}_{d^2}(X_1,\ldots,X_{\mu}; Y_1,\ldots,Y_{\mu}) = \prod_{i=1}^{\mu} \mbox{cap}_{d^2}(X_i,Y_i)$$
where $X_i =\{ x_{i,1}, \ldots, x_{i,d^2} \}$ is a Kemer polynomial of $M_d(F)$, proving that $\kappa_{M_d(F)}= (d^2,0)$.

The next natural and important example is the algebra of upper block triangular matrices $UT(d_1, \ldots, d_q)$ for positive integers $d_1, \ldots, d_q$. This is the subalgebra of $M_{d_1+\cdots+d_q}(F)$ consisting of the matrices
$$\left( \begin{array}{cclc}
M_{d_1}(F) & & & * \\
0 & \ddots & & \\
\vdots & & & \\
0 & \cdots & 0 & M_{d_q}(F)
\end{array}\right).$$
This is a basic algebra with Kemer index $(d, q-1)$, where $d=d_{1}^{2}+\cdots +d_{q}^{2}$. Indeed, it is well known that $UT(d_1, \ldots, d_q)$ has exponent $d$ and hence its Kemer index has the form $\kappa=(d,s)$. Moreover, since the nilpotency degree of $UT(d_1, \ldots, d_q)$ is $q-1$ we have $s \leq q-1$.  In order to complete the proof we will construct (Kemer) polynomials with arbitrary many small sets of cardinality $d$ and precisely $q-1$ sets of cardinality $d+1$. We start with the construction of polynomials with arbitrary many small sets of cardinality $d$. For the simple component $M_{d_i}(F)$, $i=1,\dots,q$, we consider the polynomial $\mbox{Cap}_{d_i^2}(X_{i,1},\ldots,X_{i,\mu}; Y_{i,1},\ldots,Y_{i,\mu})$ and their product bridged by $w_1,\ldots,w_{q-1}$

$$\mbox{Cap}_{d_1^2}(X_{1,1},\ldots,X_{1,\mu}; Y_{1,1},\ldots,Y_{1,\mu})\times w_1 \cdots \cdots w_{q-1} \times \mbox{Cap}_{d_q^2}(X_{q,1},\ldots,X_{q,\mu}; Y_{q,1},\ldots,Y_{q,\mu}).$$
which we denote by $\mbox{Cap}_{d_1^{2},\ldots,d_q^{2}}(X_{i,j};Y_{i,j}, i=1,\ldots,q, j=1,\ldots,\mu, W)$ or

for short $\mbox{Cap}_{d_1^{2},\ldots,d_q^{2}}(X_{i,j};Y_{i,j}, W)$.

Now for $j=1,\ldots,\mu$ we alternate in the polynomial above the sets $X_{1,j},\ldots,X_{q,j}$ and obtain a polynomial which we denote by
$f_{1, \mu}(X_{i,j};Y_{i,j}, W)$.  Finally we construct $q-1$ big sets by alternating $w_j$ with the set
$X_j= X_{1,j}\cup \cdots \cup X_{q,j}$, for $j=1,\ldots, q-1$. The result is a polynomial $f_{2,\mu}$ which alternates on $\mu - (q-1)$ small sets of cardinality $d$ and precisely $q-1$ big sets of cardinality $d +1$. We leave the reader the task to show that $f_{1,\mu}$ and $f_{2,\mu}$ are nonidentities of $UT(d_1, \ldots, d_q)$ by presenting nonzero evaluations (variables of $W$ are evaluated on radical elements).
Our construction of $f_{2, \mu}$ shows that $\kappa$, the Kemer index of $UT(d_1, \ldots, d_q)$, satisfies $\kappa_{UT(d_1, \ldots, d_q)} \geq (d, q-1)$. On the other hand $\kappa_{UT(d_1, \ldots, d_q)} \leq \Par(UT(d_1, \ldots, d_q))=(d, q-1)$ and hence $\kappa=\Par$. This shows $UT(d_1, \ldots, d_q)$ is basic and $f_2$ is a Kemer polynomial.

\subsubsection{The associated basic algebra}
As above we let $A \cong A_1 \times \cdots \times A_q   \oplus J(A)$, $r= \dim_{F}(J(A))$. For any $u \in \mathbb{N}$ consider the associated algebra
$$\mathcal{A}_{u} : = \frac{A_{ss} \ast F\{ b_1, \ldots, b_r \}}{\langle b_1, \ldots, b_r \rangle^{u+1}_{A_{ss} \ast F \{b_1, \ldots, b_r \}}}.$$

The case where $u = \mbox{nildeg}(J(A))-1$ is of special interest. We denote  $\mathcal{A}= \mathcal{A}_{\mbox{nildeg}(J(A))-1}$.

\begin{lemma}\label{finite dimension of associated algebra}The algebra $\mathcal{A}$ is finite dimensional. Moreover, if the algebra $A$ is basic then $\mathcal{A}$ is also basic.

\end{lemma}

\begin{proof}
Choose a basis $\Phi$ of $A_{ss}$ (e.g. the elementary matrices of the simple components $A_{j}$). Consider nonzero monomials in $A_{ss} \ast F\{ b_1, \ldots, b_r \}$. These are words of the form

$$
a_{i_1}b_{i_1}a_{i_2}b_{i_2}\cdots a_{i_k}b_{i_k}a_{i_{k+1}}
$$
where $k \geq 0$, $a_{i_j} \in \Phi$ and $b_{i_j} \in X$. By definition of $\mathcal{A}$, monomials are zero in $\mathcal{A}$ if $k > u$ and hence their number is finite. This proves the first part of the lemma. For the second part, note that $A_{ss}$ is a maximal semisimple subalgebra which supplements the radical $J(\mathcal{A})$. Moreover, the radical is generated by the variables $b_{i}$ and its nilpotency degree equals $nildeg(J(A))$. It follows that $\Par(\mathcal{A})=(d, u)$. But the algebra $A$ is a quotient of $\mathcal{A}$ and hence its Kemer index is at least the Kemer index of $A$. This implies the Kemer index of $\mathcal{A}$ equals $\Par(\mathcal A)=(d, u)$ and the result follows.

\end{proof}

\begin{remark}

It turns out that in the lemma above, the condition on the algebra $A$ being basic is not necessary. Let $A_{ss}$ be a semisimple algebra with $q$ simple components. It was proved by the $3$rd named author of this article that the algebra $\mathcal{A}_{u}$, $u \geq q$ is basic. This will appear in a subsequent paper authored by him.

\end{remark}

\section{Giambruno Conjecture} \label{sec: proof Giam Conj}
During this entire section we consider a basic $F$-algebra $A$ whose Wedderburn-Malcev decomposition is given by
$$
A=A_{ss}\oplus J(A),
$$
where $A_{ss} = A_1 \times \cdots \times A_q$ is a product of matrix algebras, say $A_i \cong M_{d_i}(F)$, $i=1,\ldots,q.$ Next, denote by $\kappa_A = (d,s)$ the Kemer index of $A$. In particular, by Theorem \ref{th: basis iff par= index}, $d = d^{2}_1 + \cdots + d^{2}_q$ and $s = \mbox{nildeg}(J(A)) -1$ where $\mbox{nildeg}(J(A))$ is the nilpotency degree of $J(A)$.

The proof consists of two parts, namely we show $(q-d)/2 +s$ bounds from above and from below $t(A)$. For the upper bound observe that $\id(\mathcal{A}) \subseteq \id(A)$. Also, the algebras $A$ and $\mathcal{A}$ have the same Kemer index and moreover have isomorphic semisimple subalgebras supplementing the corresponding radicals. In particular they have the same exponent. It follows that $t(A) \leq t(\mathcal{A})$ and hence it is sufficient to show $t(\mathcal{A}) \leq (q-d)/2 +s$.

\subsection{Upper bound} \label{sec: upper bound}

As remarked above we need to show $t(\mathcal{A}) \leq (q-d)/2 +s$ where $\mathcal{A}$ is the basic algebra
\[
\mathcal{A}=\frac{A_{ss}\ast F\{b_{1},\ldots,b_{\dim_{F}J(A)}\}}{\left\langle b_{1},\ldots,b_{\dim_{F}J(A)}\right\rangle ^{s+1}}.
\]

\subsubsection{Some reductions}

First, we present a preferred basis for $\mathcal{A}$. For
$1 \leq l \leq q$ we denote the matrix units of $A_{l}$ by $e_{j_{1},j_{2}}(A_{l})$
and $e_{j_{1}}(A_{l})=e_{j_{1},j_{1}}(A_{l})$, $1 \leq j_{1}, j_{2} \leq d_{l}$. Next, for $1 \leq k, k^{\prime} \leq q$, $1 \leq i \leq d_{k}$ and $1 \leq j \leq d_{k^{\prime}}$  let
\[
W_{i,j}(A_{k},A_{k^{\prime}})=\{e_{i,j_{0}}(A_{k})b_{l_{0}}e_{i_{1},j_{1}}(A_{k_{1}})b_{l_{1}}\cdots e_{i_{s^{\prime}},j_{s^{\prime}}}(A_{k_{s^{\prime}}})b_{l_{s^{\prime}}}e_{i_{s^{\prime}+1},j}(A_{k^{\prime}})|0\leq s^{\prime}\leq s\}.
\]
Note that in the expression above, the indices $j_{0}$ and $i_{s^{\prime}+1}$ run over the sets $\{1,\ldots,d_k\}$ and $\{1,\ldots,d_{k^{\prime}}\}$ respectively, the indices $i_p,j_p$ run over the set $\{1,\ldots,d_p\}$, $p = 1,\ldots,s^{\prime}$, and $l_\nu$, $\nu= 0,\ldots, s^{\prime}$, runs over the set $\{0,\ldots, dim_{F}J(A)$\}.

We denote by $W$ the union of the sets $W_{i,j}(A_{k_{1}},A_{k_{2}})$, $1 \leq k_{1}, k_{2} \leq q$, $1 \leq i \leq d_{k_1}$, $1 \leq j \leq d_{k_2}$.
Thus, \underline{a basis for $\mathcal{A}$} is the set
\[
\left\{ e_{j_{1},j_{2}}(A_{l})\,|\, 1 \leq l \leq q, \,1\leq j_{1},j_{2}\leq d_{l}\right\} \cup W.
\]

Since $\mathcal{A}$ is a finite dimensional algebra by proposition \ref{finite dimension of associated algebra}, we can identify
its relatively free algebra $F \langle X_i \mid i \in \mathbb{N} \rangle/ \Id(\mathcal{A})$ with a subalgebra of
\[
\mathcal{A}_{K}=\mathcal{A}\otimes_{F}K,
\]
where $K$ is the (commutative) polynomial ring
\[
K=F\left[\theta_{j_{1},j_{2}}^{(i)}(A_{l}),\theta^{(i)}(w)\,|\, i\in\mathbb{N},\, 1 \leq l \leq q,\,1\leq j_{1},j_{2}\leq d_{l},\, w\in W\right].
\]
Let us make this identification explicit. Write,
\[
X_{i}(A_{k})=\sum_{a_{1},a_{2}}\theta_{a_{1},a_{2}}^{(i)}(A_{k})e_{a_{1},a_{2}}(A_{k}).
\]
Then the variable $X_{i}$ of the relatively free algebra of $\mathcal{A}$ is identified with

\[
\sum_{k=1}^{q}X_{i}(A_{k})+\sum_{w\in W}\underset{X_{i}(w)}{\underbrace{\theta^{(i)}(w)w}}=\sum_{\Sigma}X_{i}(\Sigma)\in\mathcal{A}_{K},
\]
where $\Sigma$ is a symbol which runs over the set $\mathbf{Symb}=W\cup\left(\mathbf{SimComp}:=\left\{ A_{1},\ldots,A_{q}\right\} \right)$.
As a result of this identification, the spaces $P_{n}(\mathcal{A})$
are viewed as subspaces of $\mathcal{A}_{K}$.

We decompose $P_{n}(\mathcal{A})$ into subspaces as follows. Consider a monomial $X_{\sigma(1)}\cdots X_{\sigma(n)}\in P_{n}(\mathcal{A})$,
where $\sigma\in S_{n}$. Clearly, by the identification we just described, the corresponding monomial in $\mathcal{A}_K$ is equal to the sum
\begin{equation}
\sum_{\Sigma_{1},\ldots,\Sigma_{n}\in\mathbf{Symb}}X_{\sigma(1)}(\Sigma_{1})\cdots X_{\sigma(n)}(\Sigma_{n}).\label{SumOfAllPaths}
\end{equation}
Note that
\begin{enumerate}
\item $X_{i}(A_{k})X_{j}(A_{k^{\prime}})=0$ if $k\neq k^{\prime}$.
\item If more than $s$ symbols from $\Sigma_{1},\ldots,\Sigma_{n}$ are radical
(i.e. from $W$), then
\[
X_{\sigma(1)}(\Sigma_{1})\cdots X_{\sigma(n)}(\Sigma_{n})=0.
\]

\end{enumerate}
This leads to the following definition.
\begin{definition}
A sequence $\overrightarrow{p}=(p_{1},\ldots,p_{n})$ of symbols in $\mathbf{Symb}$ is called
a \emph{path} (of length $n$) if the following two properties are satisfied:
\begin{enumerate}
\item If $p_{i},p_{i+1}\in\mathbf{SimComp}$, then $p_{i}=p_{i+1}$.
\item Not more than $s$ symbols (from $p_{1},\ldots,p_{n}$) are in $W$.
\end{enumerate}

Furthermore, suppose $\overrightarrow{p}=(A_{k_{1}},\ldots,A_{k_{1}},w_{1},A_{k_{2}},\ldots,A_{k_{2}},w_{2},\ldots,w_{s^{\prime}},A_{k_{s^{\prime}+1}},\ldots,A_{k_{s^{\prime}+1}})$,
then $\mathbf{struc}(\overrightarrow{p})$, the \emph{path structure} of $\overrightarrow{p}$,
is defined to be the sequence $(A_{k_{1}},w_{1},A_{k_{2}},w_{2},\ldots,w_{s^{\prime}},A_{k_{s^{\prime}+1}})$ (i.e. we record the simple components with no adjacent repetitions and the radical elements).
Two paths $\overrightarrow{p_{1}},\overrightarrow{p_{2}}$ (of the
same length) are called \emph{equivalent }(denoted by $\overrightarrow{p_{1}}\sim\overrightarrow{p_{2}}$)
if they have the same path structure (e.g. $(A_{k_{1}}, A_{k_{1}}, w_{1}, A_{k_{2}})$ and $(A_{k_{1}}, w_{1}, A_{k_{2}}, A_{k_{2}})$ are equivalent paths whereas $(A_{k_{1}}, A_{k_{1}}, w_{1}, A_{k_{2}})$ and $(A_{k_{1}}, w_{1}, A_{k_{2}})$ are not).
\end{definition}

The set of all paths of length $n$ is denoted by $\mathbf{Path}_{n}$
and the set of all equivalence classes of paths of length $n$ is denoted by $\mathbf{Path}_{n}/\sim$.
\begin{definition}
For a given path $\overrightarrow{p} \in \mathbf{Path}_{n}$ we denote the number of appearances of
a symbol $\Sigma$ from $\mathbf{Symb}$ by $\overrightarrow{p}(\Sigma)$.
\end{definition}
By definition, the expression in  (\ref{SumOfAllPaths}) can now be rewritten as (indeed, we ignore some vanishing monomials)
\[
\sum_{\overrightarrow{p}= (p_1,\ldots,p_n)\in\mathbf{Path}_{n}}X_{\sigma(1)}(p_{1})\cdots X_{\sigma(n)}(p_{n})=\sum_{[\overrightarrow{p_{1}}]\in\mathbf{Path}_{n}/\sim}\left(\sum_{\overrightarrow{p} = (p_1,\ldots,p_n)\in[\overrightarrow{p_{1}}]}X_{\sigma(1)}(p_{1})\cdots X_{\sigma(n)}(p_{n})\right).
\]

\begin{definition}
For a fixed $\overrightarrow{p} \in \mathbf{Path}_{n}$ denote by $P_{\overrightarrow{p}}(\mathcal{A})$ the $F$-linear span
of all monomials $X_{\sigma(1)}(p_{1})\cdots X_{\sigma(n)}(p_{n})$,
where $\sigma$ varies over $S_{n}$. Furthermore, $P_{[\overrightarrow{p_1}]}(\mathcal{A})$
denotes the sum of all $P_{\overrightarrow{p}}(\mathcal{A})$
such that $\overrightarrow{p}\sim\overrightarrow{p_1}$.
\end{definition}

\begin{lemma}
\label{lem:EquiPathsApproximation}The space $P_{n}(\mathcal{A})$
is embedded in
\[
\bigoplus_{\overrightarrow{p}\in\mathbf{Path}_{n}}P_{\overrightarrow{p}}(\mathcal{A})=\bigoplus_{[\overrightarrow{p}]\in\mathbf{Path}_{n}/\sim}P_{[\overrightarrow{p}]}(\mathcal{A}).
\]
As a result,
\[
c_{n}(A)\leq\sum_{[\overrightarrow{p}]\in\mathbf{Path}_{n}/\sim}\dim_{F}P_{[\overrightarrow{p}]}(\mathcal{A}).
\]

Moreover, the size of the set $\mathbf{Path}_{n}/\sim$ is bounded
by a constant independent of $n$.
\end{lemma}
\begin{proof}
Only the last part requires an explanation. Indeed, the size of $\mathbf{Path}_{n}/\sim$
is bounded from above by the number of sequences of length at most
$2s+1$ whose elements are taken from the finite set $\mathbf{Symb}$.
So the constant can be taken to be
\[
\sum_{t=1}^{2s+1}|\mathbf{Symb}|^{t}.
\]
\end{proof}
\begin{remark}
By the previous Lemma, it is sufficient to show that $\dim_{F}P_{[\overrightarrow{p}]}(\mathcal{A})$, $\overrightarrow{p} \in \mathbf{Path}_{n}$,
is bounded from above by $Cn^{\frac{q-d}{2}+s}d^{n}$, where $C$
is some constant independent of $n$..
\end{remark}
We intend to decompose each $P_{[\overrightarrow{p}]}(\mathcal{A})$
into a (direct) sum of some special subspaces.

\begin{definition}
Let $\overrightarrow{p} = (p_1,\ldots,p_n)$ be a path in $\mathbf{Path}_{n}$ and let $Z=X_{\sigma(1)}(p_{1})\cdots X_{\sigma(n)}(p_{n})$
be a monomial in $P_{\overrightarrow{p}}(\mathcal{A})$ for some $\sigma \in S_{n}$. For $1\leq l\leq q$
we denote by $\mathbf{ind}_{l}(Z)$ the \uline{set} of all indices
$\sigma(u)$ (here $1\leq u\leq n$) for which $p_{u}=A_{l}$.

Furthermore, we denote by $\mathbf{seq}_{rad}(Z)$ the \uline{sequence}
of indices $(\sigma(i_{1}),\ldots,\sigma(i_{s^{\prime}}))$ for which
\begin{enumerate}
\item $p_{i_{u}}\in W$ for every $1\leq u\leq s^{\prime}$.
\item $i_{1}<\cdots<i_{s^{\prime}}$.
\item $\{\sigma(i_{1}),\ldots,\sigma(i_{s^{\prime}})\}\cup\mathbf{ind}_{1}(Z)\cup\cdots\cup\mathbf{ind}_{q}(Z)=\{1,\ldots,n\}$
(that is $\mathbf{seq}_{rad}(Z)$ consists of all the indices whose
corresponding variables take values in the radical).
\end{enumerate}

Finally, we set $\overrightarrow{\mbox{ind}}(Z)=(\mathbf{ind}_{1}(Z),\ldots,\mathbf{ind}_{q}(Z);\mathbf{seq}_{rad}(Z))$.

\end{definition}

\begin{definition}
Two monomials $Z_{1}$ and $Z_{2}$ in $P_{[\overrightarrow{p}]}(\mathcal{A})$
are \emph{equivalent }(or $Z_{1}\sim Z_{2}$) if $\overrightarrow{\mbox{ind}}(Z_{1})=\overrightarrow{\mbox{ind}}(Z_{2})$.
The set of all equivalence classes corresponding to this relation is denoted by $\mathbf{Mon}_{[\overrightarrow{p}]}/\sim$, where $\mathbf{Mon}_{[\overrightarrow{p}]}$ is the set of monomials in $P_{[\overrightarrow{p}]}(\mathcal{A})$. Furthermore,
$P_{[Z]}(\mathcal{A})(\subseteq P_{[\overrightarrow{p}]}(\mathcal{A}))$
denotes the $F$-span of all monomials in $P_{[\overrightarrow{p}]}(\mathcal{A})$
which are equivalent to $Z$.
\end{definition}

\begin{lemma}
\label{lem:MonomialEquivalenceApproximation}The following hold:
\begin{enumerate}
\item $P_{[\overrightarrow{p}]}(\mathcal{A})$ is equal to
\[
\bigoplus_{[Z]\in\mathbf{Mon}_{[\overrightarrow{p}]}/\sim}P_{[Z]}(\mathcal{A}).
\]

\item Denote by $\mathbf{Mon}_{[\overrightarrow{p}]}(n_{1},\ldots,n_{q})/\sim$
the subset of $\mathbf{Mon}_{[\overrightarrow{p}]}/\sim$ consisting
of all $[Z]$ for which the corresponding path $\overrightarrow{p}$
satisfies $n_{1}=|\mathbf{ind}_{1}(Z)|,\ldots,n_{q}=|\mathbf{ind}_{q}(Z)|$.
Then, $\mathbf{Mon}_{[\overrightarrow{p}]}/\sim$ is equal to the (disjoint)
union
\[
\bigcup_{n_{1}+\cdots+n_{q}= n-s'}\mathbf{Mon}_{[\overrightarrow{p}]}(n_{1},\ldots,n_{q})/\sim,
\]
where $s^{\prime} = |\mathbf{seq}_{rad}(Z)|$ is the number of symbols from $W$ in $\overrightarrow{p}$.

\item The size of $\mathbf{Mon}_{[\overrightarrow{p}]}(n_{1},\ldots,n_{q})/\sim$
is bounded from above by
\[
n^{s^{\prime}}\cdot{n-s^{\prime} \choose n_{1},\ldots,n_{q}}.
\]

 %$C$ is a constant independent of $[\overrightarrow{p}]$ and $s^{\prime} = |\mathbf{seq}_{rad}(Z)|$ the number of symbols from $W$ in $\overrightarrow{p}$.
\end{enumerate}
\end{lemma}

\begin{proof}
Only the third part requires a proof. There are $s^{\prime}!\cdot{n \choose s^{\prime}}$
options to choose and order $s^{\prime}$ indices from the set $\{1,\ldots,n\}$, i.e. there are $s^{\prime}!\cdot{n \choose s^{\prime}}$ ways to choose $\mathbf{seq}_{rad}$ for a fixed $1 \leq s^{\prime} \leq s$. From the remaining $n-s^{\prime}$
indices there are ${n-s^{\prime} \choose n_{1},\ldots,n_{q}}$ options
to choose $n_{1}$ which will correspond to $A_{1}$,\ldots, $n_{q}$
which will correspond to $A_{q}$. Finally, it is clear that
%there is a constant $C$ so that
\[
s^{\prime}!\cdot{n \choose s^{\prime}}{n-s^{\prime} \choose n_{1},\ldots,n_{q}}\leq n^{s^{\prime}}{n-s^{\prime} \choose n_{1},\ldots,n_{q}}.
\]
\end{proof}

\begin{definition} \label{most common type monomials}
%For $\mathbf{i}=(i_{1},\ldots,i_{l})$ we denote the sequence of variables $(X_{i_{1}},\ldots,X_{i_{l}})$ by $X_{\overrightarrow{\mathbf{i}}}$,
For $\mathbf{\overrightarrow{i}}=(i_{1},\ldots,i_{l})$ we denote the product $X_{i_{1}}(A_j)\cdots X_{i_{l}}(A_j)$ by $\mathbf{X_{\overrightarrow{\mathbf{i}}}}(A_j)$.

Consider monomials in $P_{[Z]}(\mathcal{A})(\subseteq P_{[\overrightarrow{p}]}(\mathcal{A}))$
of the form
\[
\mathbf{X}_{\overrightarrow{\mathbf{i}_{1}}}(A_{k_{1}})X_{\nu_{1}}(w_{1})\mathbf{X}_{\overrightarrow{\mathbf{i}_{2}}}(A_{k_{2}})X_{\nu_{2}}(w_{2})\cdots\mathbf{X}_{\overrightarrow{\mathbf{i}_{s^{\prime}+1}}}(A_{k_{s^{\prime}+1}}),
\]
namely monomials which satisfy

\begin{enumerate}
\item

$\mathbf{struc}(\overrightarrow{p})=(A_{k_{1}},w_{1},A_{k_{2}},w_{2},\ldots, A_{k_{s^{\prime}+1}})$

\item
$\bigcup\limits_{\alpha:k_{\alpha}=l}Set_{\overrightarrow{\mathbf{i}_{\alpha}}}=\mathbf{ind}_{l}(Z)$, where $Set_{x}$ consists of all indices appearing in the vector $x$

\item

$\mathbf{seq}_{rad}(Z)=(\nu_{1},\ldots,\nu_{s^{\prime}})$.
\end{enumerate}

\end{definition}
\begin{remark} \label{There are other type of monomials}
It is important to stress that there exist other type of monomials, namely monomials where some radical elements are adjacent or monomials which start or end by radical elements. As it will be clear below, the treatment of these monomials is similar to the monomials of the type considered in Definition \ref{most common type monomials}.
\end{remark}
Consider the spaces $P_{[Z]}^{j_{0},j_{s^{\prime}+1}}(\mathcal{A}):=e_{j_{0}}(A_{k_{1}})P_{[Z]}(\mathcal{A})e_{j_{s^{\prime}+1}}(A_{k_{s^{\prime}+1}})(\subseteq\mathcal{A}_{K})$,
where $1\leq j_{0}\leq d_{k_{1}}$ and $1\leq j_{s^{\prime}+1}\leq d_{k_{s^{\prime}+1}}$ (recall our notation $e_{j}(B)$, the diagonal matrix $e_{j,j}$ in the matrix algebra $B$). Note that the simple components $A_{k_1}$ and $A_{k_{s^{\prime}+1}}$ are determined by the monomial $Z$ if $P_{[Z]}^{j_{0},j_{s^{\prime}+1}}(\mathcal{A})\neq 0$ and hence we do not record the indices $k_1$ and $k_{s^{\prime}+1}$ in the definition of $P_{[Z]}^{j_{0},j_{s^{\prime}+1}}(\mathcal{A})$.
Since any element $f$ of $P_{[Z]}(\mathcal{A})$ can be written
as the sum
\[
f=1(A_{k_{1}})\cdot f\cdot1(A_{k_{s^{\prime}+1}})=\sum\limits_{\substack{1 \leq  j_{0} \leq d_{k_1} \\ 1\leq j_{s^{\prime}+1} \leq d_{k_{s^{\prime} +1}}}}e_{j_{0}}(A_{k_{1}})\cdot f\cdot e_{j_{s^{\prime}+1}}(A_{k_{s^{\prime}+1}}),
\]
we obtain an injective map
\[
P_{[Z]}(\mathcal{A})\to\bigoplus\limits_{\substack{1 \leq  j_{0} \leq d_{k_1} \\ 1\leq j_{s^{\prime}+1} \leq d_{k_{s^{\prime} +1}}}}P_{[Z]}^{j_{0},j_{s^{\prime}+1}}(\mathcal{A}).
\]

So we have proved
\begin{lemma}
\label{lem:FinalReduction}$\dim_{F}P_{[Z]}(\mathcal{A})\leq\sum\limits_{j_{0},j_{s^{\prime}+1}}\dim_{F}P_{[Z]}^{j_{0},j_{s^{\prime}+1}}(\mathcal{A})$.
\end{lemma}
As a result of this observation we will fix also the indices $j_{0}$,
$j_{s^{\prime}+1}$ and work in the space $P_{[Z]}^{j_{0},j_{s^{\prime}+1}}(\mathcal{A})$. In Lemma \ref{lem:UpperBoundKey} we describe the asymptotics of $\dim_F P_{[Z]}^{j_{0},j_{s^{\prime}+1}}(\mathcal{A})$ and then, subsequently, we will sum up all spaces of that form.

\subsubsection{The key lemma and upper bound}

%In order to proceed let us summarize our set up. We fix the path $\overrightarrow{p} \in Path_{n}$ and  numbers $n_1, \ldots, n_q$. Also $[Z]$ such that $|\mathbf{ind}_{j}(Z)| =n_j$ %is fixed.

In order to be able to carry out manipulations in the vector space
$P_{[Z]}^{j_{0},j_{s^{\prime}+1}}(\mathcal{A})$ we introduce the
following notation:

\[
\theta_{a_{1},\ldots,a_{l+1}}^{X_{1},\ldots,X_{l}}(A_{k}):=\theta_{a_{1},a_{2}}^{(X_{1})}(A_{k})\theta_{a_{2},a_{3}}^{(X_{2})}(A_{k})\cdots\theta_{a_{l},a_{l+1}}^{(X_{l})}(A_{k})
\]
and
\[
\theta_{i|j}^{X_{1},\ldots,X_{l}}(A_{k}):=\sum_{a_{2},\ldots,a_{l}}\theta_{a_{1}=i,a_{2},\ldots,a_{l},a_{l+1}=j}^{X_{1},\ldots,X_{l}}(A_{k}).
\]
Note that we changed slightly the notation we introduced above by replacing $\theta_{a_{k},a_{r}}^{(l)}$ with $\theta_{a_{k},a_{r}}^{(X_{l})}$.

Furthermore, if $\mathbf{\overrightarrow{\nu}}=(1,\ldots,l)$ we simply write
\[
\theta_{i|j}^{X_{\mathbf{\overrightarrow{\nu}}}}(A_{k})=\theta_{i|j}^{X_{1},\ldots,X_{l}}(A_{k}).
\]

The next lemma is straightforward (proof is omitted).
\begin{lemma}
\label{lem:UpperBoundKey}The following hold.
\begin{enumerate}
\item For $w_{1}\in W_{-,i}(A_{-},A_{k}),w_{2}\in W_{j,-}(A_{k},A_{-})$
we have
\[
w_{1}\mathbf{X}_{\mathbf{\overrightarrow{\nu}}}(A_{k})w_{2}=\theta_{i|j}^{X_{\mathbf{\overrightarrow{\nu}}}}(A_{k})\cdot w_{1}e_{i,j}(A_{k})w_{2}.
\]

\item For $w_{1}\in W_{j_{1},\tilde{j}_{1}}(A_{k_{1}},A_{k_{2}}),\ldots,w_{s^{\prime}}\in W_{j_{s^{\prime}},\tilde{j}_{s^{\prime}}}(A_{k_{s^{\prime}}},A_{k_{s^{\prime}+1}})$
we have
\[
\underset{w_{0}}{\underbrace{e_{\tilde{j}_{0}=j_{0}}(A_{k_{1}})}}\left(\mathbf{X}_{\overrightarrow{\mathbf{\nu}_{1}}}(A_{k_{1}})X_{i_{1}}(w_{1})\mathbf{X}_{\overrightarrow{\mathbf{\nu}_{2}}}(A_{k_{2}})X_{i_{2}}(w_{2})\cdots\mathbf{X}_{\overrightarrow{\mathbf{\nu}_{s^{\prime}+1}}}(A_{k_{s^{\prime}+1}})\right)\underset{w_{s^{\prime}+1}}{\underbrace{e_{\tilde{j}_{s^{\prime}+1}=j_{s^{\prime}+1}}(A_{k_{s^{\prime}+1}})}}
\]
 equals to
\[
\theta^{(i_{1})}(w_{1})\cdots\theta^{(i_{s'})}(w_{s^{\prime}})\left(\prod_{l=1}^{q}\prod_{\alpha:k_{\alpha}=l}\theta_{\tilde{j}_{\alpha-1}|j_{\alpha}}^{X_{\mathbf{\overrightarrow{\nu_{\alpha}}}}}(A_{l})\right) \cdot \mathbf{w},
\]
where  $$\mathbf{w} = w_{0}e_{\tilde{j}_{0},j_{1}}(A_{k_{1}})w_{1}e_{\tilde{j}_{1},j_{2}}(A_{k_{2}})\cdots w_{s^{\prime}}e_{\tilde{j}_{s^{\prime}},j_{s^{\prime}+1}}(A_{k_{s^{\prime}+1}})w_{s^{\prime}+1}.$$

\end{enumerate}

\end{lemma}

\begin{corollary}
There exists an element $w \in A$ such that for any monomial $Z^{\prime} \in P_{[Z]}^{j_{0},j_{s^{\prime}+1}}(\mathcal{A})$  $($and hence for any element of $P_{[Z]}^{j_{0},j_{s^{\prime}+1}}(\mathcal{A})$$)$
$$
Z^{\prime} = f\cdot w
$$
where $f$ is a polynomial in $F\left[\theta_{j_{1},j_{2}}^{(i)}(A_{l}),\theta^{(i)}(w)\,|\, i\in\mathbb{N},\, 1 \leq l \leq q,\,1\leq j_{1},j_{2}\leq d_{l},\, w\in W\right].$
\end{corollary}

We now turn to the construction of the map which enables to estimate $\dim_{F}P_{[Z]}^{j_{0},j_{s^{\prime}+1}}$.

For every $l$, consider the variables $Y_{1,l},\ldots,Y_{v_l-1,l}$, where
$v_{l}$ is the number of appearances of $A_{l}$ in $\mathbf{struc}(\overrightarrow{p})=(A_{k_{1}},w_{1},A_{k_{2}},w_{2},\ldots, A_{k_{s^{\prime}+1}})$.

Let $P_{X_{h_1},\ldots,X_{h_{n_l}}; Y_{1,l},\ldots,Y_{v_{l}-1,l}}$ denote the $(n_{l} + v_{l} - 1)!$ space of all multilinear polynomials in the prescribed variables, where $\mathbf{ind}_{l}(Z) = (h_1,\ldots,h_{n_l})$ (i.e. the indices of variables which get values from $A_{l}$). Let

$$
P_{X_{h_1},\ldots,X_{h_{n_l}}; Y_{1,l},\ldots,Y_{v_{l}-1,l}}(A_{l}) = P_{X_{h_1},\ldots,X_{h_{n_l}}; Y_{1,l},\ldots,Y_{v_{l}-1,l}}/(P_{X_{h_1},\ldots,X_{h_{n_l}}; Y_{1,l},\ldots,Y_{v_{l}-1,l}} \cap Id(A_{l})).
$$

Denote $$(U_{1},\ldots,U_{n_{l} + v_{l} + 1})=(X_{h_1},\ldots,X_{h_{n_l}}, Y_{1,l},\ldots,Y_{v_{l}-1,l})$$ and let $P_{[\widehat{X_{A_{l}}}]}(A_{l}) = P_{[\widehat{X_{A_{l}}}]}/(P_{[\widehat{X_{A_{l}}}]} \cap Id(A_{l}))$ be the subspace of $P_{X_{h_1},\ldots,X_{h_{n_l}}; Y_{1,l},\ldots,Y_{v_{l}-1,l}}(A_{l})$ spanned by all monomials $U_{\tau(1)}U_{\tau(2)}\cdots U_{\tau(n_{l} + v_{l} - 1)}$ where

\begin{enumerate}
\item
$U_{\tau(1)}, U_{\tau(n_{l} + v_{l} - 1)} \in \{ X_{h_1},\ldots,X_{h_{n_l}}\}$

\item

If $i < j$ and $U_{\tau(i)}, U_{\tau(j)} \in \{Y_{1,l},\ldots,Y_{v_{l}-1,l}\}$, then there exists $k$ with $i < k < j$ such that $U_{\tau(k)} \in
\{ X_{h_1},\ldots,X_{h_{n_l}}\}.$

\item

The ordering induced by $\tau$ on the set $\{Y_{1,l},\ldots,Y_{v_{l}-1,l}\}$ is precisely the ordering $(Y_{1,l},\ldots,Y_{v_{l}-1,l}).$ That is, for
$n_l +1 \leq i,j \leq n_{l} + v_{l} - 1$ we have $ i < j \Longrightarrow \tau(i) < \tau(j).$

\end{enumerate}

\begin{remark}
We abuse notation and language here by considering monomials as elements of $P_{[\widehat{X_{A_{l}}}]}(A_{l})$
\end{remark}

\begin{notation}
We denote by $c_{\widehat{X_{A_{l}}}}(A_{l}) = dim_{F}P_{[\widehat{X_{A_{l}}}]}(A_{l})$ and note that $c_{\widehat{X_{A_{l}}}}(A_{l}) \leq c_{n_{l}+v_{l} -1}(A_{l})$.
\end{notation}

Now, for $l=1,\ldots,q$ let

$$\widehat{X_{A_{l}}} = X_{\overrightarrow{\mathbf{\mu}_{\alpha(1,l)}}}\cdot Y_{1,l} \cdot X_{\overrightarrow{\mathbf{\mu}_{\alpha(2,l)}}}\cdot Y_{2,l} \cdots Y_{v_l-1,l}\cdot X_{\overrightarrow{\mathbf{\mu}_{\alpha(v_{l},l)}}},$$
be a monomial in $P_{[\widehat{X_{A_{l}}}]}(A_{l})$

and let

$$X=(\widehat{X_{A_{1}}},\ldots,\widehat{X_{A_{q}}}) \in  P_{[\widehat{X_{A_{1}}}]} \times \cdots \times P_{[\widehat{X_{A_{q}}}]}.$$

Next consider the monomial

\[e_{\tilde{j}_{0}=j_{0}}(A_{k_{1}})X (\mathcal{A})e_{\tilde{j}_{s^{\prime}+1}=j_{s^{\prime}+1}}(A_{k_{s^{\prime}+1}}) \in
P_{[Z]}^{j_{0},j_{s^{\prime}+1}}(\mathcal{A}),\]

where

\[X (\mathcal{A})=
\mathbf{X}_{\overrightarrow{\mathbf{\nu}_{\alpha(t_1,k_1)}}}(A_{k_{1}})X_{i_{1}}(w_{1})\mathbf{X}_{\overrightarrow{\mathbf{\nu}_{\alpha(t_2,k_2)}}}(A_{k_{2}})X_{i_{2}}(w_{2})\cdots\mathbf{X}_{\overrightarrow{\mathbf{\nu}_{\alpha(t_{s^{\prime}+1},k_{s^{\prime}+1})}}}(A_{k_{s^{\prime}+1}})
\]

and if $k_{g_1} = k_{g_2} = \ldots = k_{g_{v_l}} = l$ are the indices where the simple component $A_{l}$ appears in $\mathbf{struc}(\overrightarrow{p})=(A_{k_{1}},w_{1},A_{k_{2}},w_{2},\ldots, A_{k_{s^{\prime}+1}})$, then

\[({\overrightarrow{\mathbf{\nu}_{\alpha(t_{g_1},k_{g_1})}}},\dots,{\overrightarrow{\mathbf{\nu}_{\alpha(t_{g_{v_{l}}},k_{g_{v_l}})}}})
 = ({\overrightarrow{\mathbf{\mu}_{\alpha(1,l)}}},\dots,{\overrightarrow{\mathbf{\mu}_{\alpha(v_{l},l)}}}).\]

\begin{lemma} \label{well defined map}

The following hold.

\begin{enumerate}
\item

The map $\Psi: P_{[\widehat{X_{A_{1}}}]} \times \cdots \times P_{[\widehat{X_{A_{q}}}]} \rightarrow P_{[Z]}^{j_{0},j_{s^{\prime}+1}}(\mathcal{A})$

$$X\mapsto e_{\tilde{j}_{0}=j_{0}}(A_{k_{1}})X (\mathcal{A})e_{\tilde{j}_{s^{\prime}+1}=j_{s^{\prime}+1}}(A_{k_{s^{\prime}+1}})$$ is well defined, surjective and multilinear. Hence it determines a surjection

\[
P_{\widehat{[X_{A_{1}}]}}(A_{1})\otimes\cdots\otimes P_{\widehat{[X_{A_{q}}]}}(A_{q})\to P_{[Z]}^{j_{0},j_{s^{\prime}+1}}(\mathcal{A}).
\]

\item $\dim_{F}P_{[Z]}^{j_{0},j_{s^{\prime}+1}}(\mathcal{A})\leq c_{n_{1}+s}(A_{1})\cdots c_{n_{q}+s}(A_{q})\leq C\cdot c_{n_{1}}(A_{1})\cdots c_{n_{q}}(A_{q})$,
where $C$ is a constant which is independent of $n_{1},\ldots,n_{q}$.

\end{enumerate}

\end{lemma}

\begin{proof}
Suppose $f$ is a linear combination of monomials in $P_{[\widehat{X_{A_{l}}}]}(A_{l})$ which represents the zero element. Clearly, $f$ represents the zero map in $Hom(A_{l}^{\otimes (n_{l} + v_{l}-1)}, A_{l})$ and hence, evaluating $Y_{1,l},\ldots,Y_{v_{l}-1,l}$ on $A_{l}$ we obtain the zero map in $Hom(A_{l}^{\otimes n_{l}}, A_{l})$. In particular we obtain zero
if we evaluate
$$Y_{1,l} = e_{j_{g_1},\tilde{j}_{g_2}-1}(A_{l}),Y_{2,l} = e_{j_{g_2},\tilde{j}_{g_3}-1}(A_{l}),\ldots,Y_{v_{l}-1,l} = e_{j_{g_{v_l}-1},\tilde{j}_{g_{v_l}}-1}(A_{l}).$$

But in view of the fact that $$w_{g_1} \in W_{j_{g_1},\tilde{j}_{g_1}}(A_{k_{g_1}},A_{k_{g_1+1}}), w_{g_2} \in W_{j_{g_2},\tilde{j}_{g_2}}(A_{k_{g_2}},A_{k_{g_2+1}}),\ldots,$$
$$w_{g_{v_l-1}} \in W_{j_{g_{v_l-1}},\tilde{j}_{g_{v_l}}}(A_{k_{g_{v_l-1}}},A_{k_{g_{v_l}}}),w_{g_{v_l-1}} \in W_{j_{g_{v_l-1}},\tilde{j}_{g_{v_l}}}(A_{k_{g_{v_l-1}}},A_{k_{g_{v_l}}}),$$ and by Lemma \ref{lem:UpperBoundKey}, we see that up to a scalar the bridge between the $i$th and the $(i+1)$th appearance of $A_{l}$ in $X(\mathcal{A})$ is given precisely by the matrix $e_{j_{g_i},\tilde{j}_{g_{i+1}}-1}(A_{l})$ and hence
$e_{\tilde{j}_{0}=j_{0}}(A_{k_{1}})X (\mathcal{A})e_{\tilde{j}_{s^{\prime}+1}=j_{s^{\prime}+1}}(A_{k_{s^{\prime}+1}}) = 0$. This shows the map $\Psi$ is well defined. It is clear by construction that $\Psi$ is multilinear and onto.

Let us prove the second part. Applying part ($1$) and the inequalities $c_{\widehat{X_{A_{l}}}}(A_{l}) \leq c_{n_{l}+v_{l} -1}(A_{l})$ we have
\[
\dim_{F}P_{[Z]}^{j_{0},j_{s^{\prime}+1}}(\mathcal{A})\leq c_{\widehat{X_{A_{1}}}}(A_{1}) \cdots c_{\widehat{X_{A_{q}}}}(A_{q}) \leq c_{n_{l}+v_{l} -1}(A_{l}) \cdots c_{n_{l}+v_{l} -1}(A_{l}).
\]
Furthermore, since the number of $Y$'s is bounded by $s$, and the sequence
$c_{n}(A_{l})$ is an eventually nondecreasing function in $n$ \cite{GiaZaiEve}, we obtain
\[
\dim_{F}P_{[Z]}^{j_{0},j_{s^{\prime}+1}}(\mathcal{A})\leq c_{n_{1}+s}(A_{1})\cdots c_{n_{q}+s}(A_{q}).
\]
The last inequality in the lemma follows from the fact that $c_n(B) \simeq \Theta( n^t d^n)$ for PI algebras $B$ (as proved by Berele and Regev \cite{BerRegAsymp} for unital algebras and later by Giambruno and Zaicev \cite{GiaZaiEve} for arbitrary algebras). Indeed, we have that
\[
\lim_{n\to\infty}\frac{c_{n+s}(A_{l})}{c_{n}(A_{l})}=K_2
\]
for some constant $K_2 \in \mathbb{R}$ and the result follows. This completes the proof of the lemma.
 \end{proof}
As mentioned in Remark \ref{There are other type of monomials} other type of monomials $Z$ should be considered. The proofs of the statements that correspond to Lemmas (\ref{lem:FinalReduction} - \ref{well defined map}) are similar and left to the reader.

\begin{theorem}[Upper bound]\label{Upper bound result}
\label{thm:UpperBound} There is a constant $C$ such that
\[
c_{n}(A)\leq C\cdot n^{\frac{q-d}{2}+s}d^{n}.
\]
\end{theorem}
\begin{proof}
By part (5) of \lemref{UpperBoundKey} it follows that
\[
\dim_{F}P_{[Z]}^{j_{0},j_{s^{\prime}+1}}(\mathcal{A})\leq C_{1}\cdot c_{n_{1}}(A_{1})\cdots c_{n_{q}}(A_{q}),
\]
where $n_{1},\ldots,n_{q}$ are determined by the path corresponding
to $[Z]$. Combining this with \lemref{FinalReduction} it gives
\[
\dim_{F}P_{[Z]}(\mathcal{A})\leq C_{2}\cdot c_{n_{1}}(A_{1})\cdots c_{n_{q}}(A_{q}).
\]

By \lemref{MonomialEquivalenceApproximation} it follows that
\[
\sum_{[Z]\in\mathbf{Mon}_{[\overrightarrow{p}]}(n_{1},\ldots,n_{q})/\sim}\dim_{F}P_{[Z]}(\mathcal{A})\leq C_{3}\cdot n^{s^{\prime}}{n-s^{\prime} \choose n_{1},\ldots,n_{q}}\cdot c_{n_{1}}(A_{1})\cdots c_{n_{q}}(A_{q}),
\]
where $s^{\prime}=n-n_{1}-\cdots-n_{q}$. Thus,
\[
\dim_{F}P_{[\overrightarrow{p}]}(\mathcal{A})\leq C_{4}\cdot n^{s^{\prime}}\cdot\sum_{n_{1}+\cdots+n_{q}=n-s^{\prime}}{n-s^{\prime} \choose n_{1},\ldots,n_{q}}c_{n_{1}}(A_{1})\cdots c_{n_{q}}(A_{q}).
\]
By \lemref{EquiPathsApproximation} we obtain
\[
c_{n}(A)\leq C_{5}\cdot\sum_{s^{\prime}=0}^{s}\left(n^{s^{\prime}}\cdot\sum_{n_{1}+\cdots+n_{q}=n-s^{\prime}}{n-s^{\prime} \choose n_{1},\ldots,n_{q}}c_{n_{1}}(A_{1})\cdots c_{n_{q}}(A_{q})\right).
\]
Next, by \cite{RegAsymp},
\[
\sum_{n_{1}+\cdots+n_{q}=n-s^{\prime}}{n-s^{\prime} \choose n_{1},\ldots,n_{q}}c_{n_{1}}(A_{1})\cdots c_{n_{q}}(A_{q}) \leq C_{6}\cdot\sum_{n_{1}+\cdots+n_{q}=n-s^{\prime}}{n-s^{\prime} \choose n_{1},\ldots,n_{q}}n_{1}^{\frac{1-d_{1}^{2}}{2}}d_{1}^{2n_{1}}\cdots n_{q}^{\frac{1-d_{q}^{2}}{2}}d_{q}^{2n_{q}},
\] for some constant $C_{6}$.

This, by a theorem of Regev and Beckner (which we recall below for the reader convenience), is asymptotically equal to

\[
C_{7}\cdot(n-s^{\prime})^{\frac{q-d}{2}}d^{n}.
\]
 All in all we have
\[
c_{n}(A)\leq C_{8} \cdot\sum_{s^{\prime}=0}^{s}\left(n^{s^{\prime}}\cdot(n-s^{\prime})^{\frac{q-d}{2}}d^{n}\right)  \leq C\cdot n^{\frac{q-d}{2}+s}d^{n}
\]

as desired.
\end{proof}

\begin{theorem}$($Regev and Bekner $\cite{BecReg}$$)$ \label{RegevBeckner}
Let $r_{1},\ldots ,r_{q},k_{1}, \ldots ,k_{q}\in\mathbb{R}$
  such that $0<k_{1},\ldots ,k_{q}$.
  Then,
 $$\sum_{n_{1}+\cdots+n_{q}=n}{n \choose n_{1},\ldots ,n_{q}}k_{1}^{n_{1}}\cdots k_{q}^{n_{q}}n_{1}^{r_{1}}\cdots n_{q}^{r_{q}}\sim\left(\left(\frac{k_{1}}{k}\right)^{r_{1}}\cdots\left(\frac{k_{q}}{k}\right)^{r_{q}}\right)\cdot n^{r}k^{n},$$
 where $k=k_{1}+\cdots+k_{q}$
  and $r=r_{1}+\cdots+r_{q}$.
\end{theorem}

\subsection{Lower bound}\label{sec: lower bound}

Let $A$ be a basic algebra over a field $F$. Write $A=\left(A_{ss}=A_{1}\times\cdots\times A_{q}\right)\oplus J$
and denote by $d_{i}^{2}$ the dimension of $A_{i}$, $i=1,\ldots,q$. Furthermore,
$d=d_{1}^{2}+\cdots+d_{q}^{2}$ and $s$ denotes the nilpotency degree
of $J$ minus $1$.

\begin{convention}
In the sequel, as we may by linearity, all evaluations we consider of multilinear polynomials are from $A_{1}\cup\cdots\cup A_{q}\cup J$.
\end{convention}

Since $A$ is basic, by Theorem \ref{th: basis iff par= index}, it possesses a multilinear polynomial $$f_{0}=f_{0}\left(z_{1},\ldots,z_{q};B:=B_{1}\cup\cdots\cup B_{s};E\right),$$
where
\begin{enumerate}
\item $|B_{1}|=\cdots=|B_{s}|=d+1$.

\item $f_{0}$ alternates on each set $B_{i}$, $i=1,\ldots,s$. Therefore, any nonzero
evaluation of the variables of $f_{0}$ takes exactly one variable of every $B_{i}$
to a radical element and the remaining variables (including the $z$'s
and the ones from $E$) to a semisimple element.

\item There is a nonzero evaluation of the variables of $f_{0}$ such that
$\tilde{z}_{i}=1(A_{i})$ for $i=1,\ldots,q$.
\end{enumerate}
Till the end of this section we fix a nonzero evaluation, denoted by $\tilde{f_{0}} = f_{0}\left(\tilde{z}_{1},\ldots,\tilde{z}_{q};\tilde{B};\tilde{E}\right)$, which satisfies (3). Moreover,
for $i=1,\ldots,s$, we denote by $w_{i}\in B_{i}$ the variable such that
$\tilde{w}_{i}\in J$.

\begin{remark}
In the sequel we'll consider partial evaluations of multilinear polynomials. The properties of these evaluations rely on the existence of the evaluation $\tilde{f_{0}}$ of $f_{0}$.
\end{remark}

Consider the multilinear polynomial
\[
f_{1}=f_{1}(z_{1},\ldots,z_{q};B;Y;E)=f_{0}(y_{1,1}y_{1,2}z_{1}y_{1,3}y_{1.4},\ldots,y_{q,1}y_{q,2}z_{q}y_{q,3}y_{q,4}; B; E),
\]
where $Y=\{y_{i,j}|\, i=1,\ldots,q,\, j=1,\ldots,4\}.$ We shall abuse notation
by writing $f_{1}(z_{1},\ldots,z_{q})$ (omitting $E,Y$ and $B$). Furthermore,
we denote $B\cup Y\cup E$ by $BYE$. Note that since $f_{0}$ is a nonidentity of $A$, $f_{1}$ is a nonidentity as well.

\begin{remark}In what follows, roughly speaking, we shall replace the variables $z_1,\ldots,z_{q}$, with multilinear polynomials $g_{1},\ldots,g_{q}$. The polynomials $g_{1},\ldots,g_{q}$ will be elements in $F\langle x_{1},\ldots, x_{n} \rangle$, $n \in \mathbb{N}$, where different polynomials depend on disjoint sets of variables. This will give rise to a multilinear polynomial in the variables $\{x_{1},\ldots,x_{n}\} \cup BYE$.
\end{remark}

For any $n \in \mathbb{N}$ let $X_{n}=\{x_{1},\ldots,x_{n}\}$, and let $X_{n} \cup BYE$ be the corresponding set of variables.
Let $P_{X_{n};BYE}$ be the $F$-space of all multilinear polynomials on $X_{n} \cup BYE$. The symmetric group $S_{n}$ acts on $P_{X_{n};BYE}$ by
\[
\sigma\cdot f(x_{1},\ldots,x_{n};BYE)=f(x_{\sigma^{-1}(1)},\ldots,x_{\sigma^{-1}(n)};BYE)
\]
and it is well known that this action induces an $S_{n}$-module structure on the space
\[
P_{X_{n};BYE}(A)=\frac{P_{X_{n};BYE}}{P_{X_{n};BYE}\cap id(A)}.
\]

Now, consider
a partition $\mathbf{p}$ of the set $X_{n}$ into $q$ subsets, denoted
by $X[A_{1}],\ldots,X[A_{q}]$, where each $X[A_{i}]$ is nonempty (we are interested in $n \rightarrow \infty$, so we may assume that $n\geq q$).

Consider the symmetric groups $S_{X[A_{1}]},\ldots,S_{X[A_{q}]}$ and
their product $S_{\boldsymbol{p}}=S_{X[A_{1}]}\times\cdots\times S_{X[A_{q}]}<S_{n}$.
Clearly, by restriction, we obtain $S_{\boldsymbol{p}}$-modules structures on $P_{X_{n};BYE}$ and on
$P_{X_{n};BYE}(A).$

By the well known embedding of the relatively free algebra
of $A$ in $U_{A}=A\otimes_{F}F(\theta_{i,j}\,:\, i\in\mathbb{N},\, j=1,\ldots,\dim_{F}A)$
we may view the space $P_{X_{n};BYE}(A)$ as
a subspace of $U_{A}$. This allows us to consider partial evaluation which is a key idea in this section.

\begin{definition} \label{partial evaluations}

For $\mathbf{\tilde{j}}, \mathbf{j}\in\{1,\ldots,d_{1}\}\times\cdots\times\{1,\ldots,d_{q}\}$, let
$P_{X_{n};BYE}^{\tilde{\mathbf{j}},\mathbf{j}}(A)\subseteq U_{A}$
be the space obtained from $P_{X_{\mathbf{n}};BYE}(A)$ by performing the following evaluations on some of the variables:

\begin{enumerate}
\item

$y_{i,2}\to\bar{y}_{i,2}=e_{\tilde{j}_{i}}(A_{i})$ and $y_{i,3}\to\bar{y}_{i,3}=e_{j_{i}}(A_{i})$
for $i=1,\ldots,q$.

\item
For $i=1,\ldots,s$, and for any $w \in B_{i}\setminus \{w_{i}\}$ we replace $w$ by its value $\tilde{w}$. Note that $\tilde{w}$ is necessary a semisimple element.

\end{enumerate}

For $g\in P_{X_{n};BYE}(A)$ we denote by $\bar{g}$
its image in $P_{X_{n};BYE}^{\tilde{\mathbf{j}},\mathbf{j}}(A)$.

\end{definition}

\begin{notation}

In order to simplify our notation for elements in $P_{X_{n};BYE}^{\tilde{\mathbf{j}},\mathbf{j}}(A)$ we shall write $\overline{f(x_1,\ldots,x_n)}$ if the variables $BYE$ do not play a role in our variable manipulations. In Lemma \ref{including the radical variables} we'll need to manipulate the variables of $X_{n}$ and $B$ and so we'll write $\overline{f(x_1,\ldots,x_n;w_1,\ldots,w_s)}$.
\end{notation}

Observe that $\dim_{F}P_{X_{n};BYE}^{\tilde{\mathbf{j}},\mathbf{j}}(A)\leq\dim_{F}P_{X_{n};BYE}(A)$.
Therefore, for the lower bound, it will be sufficient to bound from below one of the
spaces $P_{X_{n};BYE}^{\tilde{\mathbf{j}},\mathbf{j}}(A)$.

The elements of $P_{X_{n};BYE}^{\tilde{\mathbf{j}},\mathbf{j}}(A)$,
as the elements of $P_{X_{n};BYE}(A)$, will be referred as \emph{multilinear
polynomials}.

As for the elements of $P_{X_{\mathbf{n}};BYE}(A)$, we perform the partial evaluations $1$ and $2$ in Definition \ref{partial evaluations} also on the polynomial
\[
f_{1}=f_{1}(z_{1},\ldots,z_{q};B;Y;E)=f_{0}(y_{1,1}y_{1,2}z_{1}y_{1,3}y_{1.4},\ldots,y_{q,1}y_{q,2}z_{q}y_{q,3}y_{q,4}; B; E),
\] and denote the obtained polynomial by $\overline{f_{1}}$.

\begin{lemma}
\label{lem:Consider-a-nonzero} For any nonzero evaluation of
$\overline{f_{1}}$ the following hold.
\begin{enumerate}
\item The variables $\{z_{1},\ldots,z_{q}\}\cup Y\cup E$ are all evaluated
by semisimple elements.

\item Each $w_{i} \in B_{i}$, $i=1,\ldots,s$, is evaluated by a radical element
(note the other elements of $B_{i}$ were already replaced by semisimple elements).

\item For every $i=1,\ldots,q$, the variable $z_{i}$ is evaluated by an
element of $A_{i}$.
\end{enumerate}

Moreover, any nonzero evaluation  $\bar{g}_{i}$ of a polynomial $g_i$  by elements
of $A_{i}$, $i=1,\ldots,q$, such that $e_{j_i}\bar{g}_{i}e_{\tilde{j}_i}\neq 0$ may be \uline{extended} to a nonzero
evaluation of $\overline{f_{\mathbf{1}}(g_{1},\ldots,g_{q})}$.

\end{lemma}
\begin{proof}
Parts (1) and (2) are clear. Part (3) now follows, since $\bar{y}_{i,2}$ (and $\bar{y}_{i,3}$)
is an element from $A_{i}$ and $z_{i}$ must be a semisimple element
(by (1)).

We turn to the last part of the Lemma. Since $\bar{y}_{i,2}\bar{g_i}\bar{y}_{i,3}\neq0$,
we can find suitable evaluations of $y_{i,1}$ and $y_{i,4}$ so that
the expression $y_{i,1}\bar{y}_{i,2}\bar{g_i}\bar{y}_{i,3}y_{i,4}$
is evaluated by any element of $A_{i}$. We are done since there is
a nonzero evaluation of $f_{0}$ were each $z_{i}$ (as variable
of $f_{0}$) is evaluated by an element of $A_{i}$.\end{proof}
\begin{lemma}
\label{lem:HomoToFunctionals}Let $V$ and $W$ be finite dimensional
vector spaces and let $U$ be a subspace of $Hom_{F}(V,W)$. Fix a
basis $w_{1},\ldots,w_{\dim_{F}W}$ of $W$. Then there is an element
$\psi$ in the dual basis of $W$ such that
\[
\dim_{F}(\psi\circ U)\geq\frac{\dim_{F}U}{\dim_{F}W}
\] where $\psi\circ U$ is the space of all elements $\psi \circ T$ where $T \in U$.
\end{lemma}
\begin{proof}
It is clear that the map $\Psi:Hom_{F}(V,W)\to V^{\ast}\oplus\cdots\oplus V^{\ast}$
($\dim_{F}W$ times) given by
\[
\Psi(T)=(\psi_{1}\circ T,\ldots,\psi_{\dim_{F}W}\circ T),
\]
is injective (in fact, it is an isomorphism), where $\psi_{1},\ldots,\psi_{\dim_{F}W}$
is any basis of $W^{\ast}$ (in our case it is the dual basis of ($w_{1},\ldots,w_{\dim_{F}W}$)). As
a result,
\[
\dim_{F}U\leq\sum_{i=1}^{\dim_{F}W}\dim_{F}\psi_{i}\circ U.
\]
So there is some $i_{0}$ such that
\[
\dim_{F}U\leq\dim_{F}W\cdot\dim_{F}\psi_{i_{0}}\circ U.
\]
 So $\psi=\psi_{i_{0}}$ is the required dual basis element.
\end{proof}
Next, consider the spaces $T_{X[A_{i}]}(A_{i})=Hom_{F}(A_{i}^{\otimes n_{i}},A_{i})$
where $i=1,\ldots,q$ ($n_i$ is the number of elements of $X[A_i]$ in the partition of $n$). Furthermore, we may view the space $P_{X[A_{i}]}(A_{i})$ as a subspace of $T_{X[A_{i}]}(A_{i})$ via the embedding $g\to \eta(g)$ which is determined by
\[
\eta(g)(a_{1}\otimes\cdots\otimes a_{n_{i}})=g(a_{1},\ldots,a_{n_i}),(a_{1},\ldots,a_{n_i}) \in A_{i}^{\otimes n_{i}}.
\]

We apply the above lemma in the following setup: $V=A_{i}^{\otimes n_{i}},W=A_{i}$
and $U=P_{X[A_{i}]}(A_{i})\subseteq T_{X[A_{i}]}(A_{i})$. We use
the basis of elementary matrices as a basis of $W$.

Now, for $i=1,\ldots,q$, we denote by $\psi_{\tilde{j}_{i},j_{i}}$ the element in the dual basis of $W=A_{i}$ as given by Lemma \ref{lem:HomoToFunctionals}. Note that $\psi_{\tilde{j}_{i},j_{i}}:A_{i}\to F$
is the map assigning to each matrix of $A_{i}$ its $e_{\tilde{j}_{i},j_{i}}$
coefficient. For the given $\mathbf{j}=(j_{1},\ldots,j_{q})$ and $\tilde{\mathbf{j}}=(\tilde{j}_{1},\ldots,\tilde{j}_{q})$ we simplify our notation and denote
the space $P_{X_{n};BYE}^{\tilde{\mathbf{j}},\mathbf{j}}(A)$ by $\bar{P}_{X_{n};BYE}(A).$

\begin{theorem}

The mapping
\[
\phi_{\mathbf{p}}:P_{X[A_{1}]}(A_{1})\otimes\cdots\otimes P_{X[A_{q}]}(A_{q})\to\bar{P}_{X_{n};BYE}(A)
\]
 given by
\[
\phi_{\mathbf{p}}(g_{1}\otimes\cdots\otimes g_{q})=\overline{f_{1}(g_{1},\ldots,g_{q})},
\]
 is well defined.

Moreover, if we denote by $M(\mathbf{p})$ the image of $\phi_{\mathbf{p}}$, then
\[
\dim M(\mathbf{p})\geq\frac{1}{d_{1}^{2}\cdots d_{q}^{2}}\cdot c_{n_{1}}(A_{1})\cdots c_{n_{q}}(A_{q}),
\]
where $n_{1}=|X[A_{1}]|,\ldots,n_{q}=|X[A_{q}]|$.\end{theorem}
\begin{proof}
It is convenient to introduce the following auxiliary spaces which we denote by

$$\bar{P}_{z_1,\ldots,z_{i-1},X[A_{i}],z_{i+1},\ldots, z_{q};BYE}(A)=P^{\tilde{\mathbf{j}},\mathbf{j}}_{z_1,\ldots,z_{i-1},X[A_{i}],z_{i+1},\ldots, z_{q};BYE}(A).$$

 As for the construction of $\bar{P}_{X_{n};BYE}(A) = P_{X_{n};BYE}^{\tilde{\mathbf{j}},\mathbf{j}}(A)$ above, the space $\bar{P}_{z_1,\ldots,z_{i-1},X[A_{i}],z_{i+1},\ldots, z_{q};BYE}(A)$, $i=1,\ldots,q$, is obtained from the space
$$P_{z_1,\ldots,z_{i-1},X[A_{i}],z_{i+1},\ldots, z_{q};BYE}(A)$$ by performing evaluations $1$ and $2$ of Definition \ref{partial evaluations}.

Note that
$$\overline{f_{1}|_{z_{i}\to g_{i}}} \in \bar{P}_{z_1,\ldots,z_{i-1},X[A_{i}],z_{i+1},\ldots, z_{q};BYE}(A).$$
Furthermore the map $\phi_{i}:P_{X[A_{i}]}(A_{i})\to \bar{P}_{z_1,\ldots,z_{i-1},X[A_{i}],z_{i+1},\ldots, z_{q};BYE}(A)$
given by $\phi_{i}(g_i)=\overline{f_{1}|_{z_{i}\to g_{i}}}$ is well
defined, since in any nonzero evaluation of $\overline{f_{1}}$ all
variables of $X[A_{i}]$ must be evaluated by elements of $A_{i}$
(see \lemref{Consider-a-nonzero}), so for $g_{i}\in Id(A_{i})$
there is no nonzero evaluation of $\overline{f_{1}|_{z_{i}\to g_{i}}}$.
In other words,
\[
g_{i}\in id(A_{i})\Longrightarrow\phi_{i}(g_i)=0.
\]
The same argument shows that the map
\[
\phi_{\mathbf{p}}^{\prime}:P_{X[A_{1}]}(A_{1})\times\cdots\times P_{X[A_{q}]}(A_{q})\to\bar{P}_{X_{n};BYE}(A)
\]
given by
\[
\phi_{\mathbf{p}}^{\prime}(g_{1},\ldots,g_{q})=\overline{f_{1}(g_{1},\ldots,g_{q})}
\]
is well defined. Since $\phi_{\mathbf{p}}^{\prime}$ is multilinear,
it induces the map
\[
\phi_{\mathbf{p}}:P_{X[A_{1}]}(A_{1})\otimes\cdots\otimes P_{X[A_{q}]}(A_{q})\to\bar{P}_{X_{n};BYE}(A).
\]

Suppose $\psi_{\tilde{j}_{i},j_{i}}\circ g_{(A_{i},1)},\ldots,\psi_{\tilde{j}_{i},j_{i}}\circ g_{(A_{i},t_i)}\in\left(A_{i}^{\otimes n_{i}}\right)^{\ast}$
is a basis of $\psi_{\tilde{j}_{i},j_{i}}\circ P_{X[A_{i}]}(A_{i})$, $i=1,\ldots,q$. \\

Applying the lemma it is clear now that in order to complete the proof it is enough to prove that
$\phi_{\mathbf{p}}$ is injective when restricted to the subspace
$T$ spanned by $g_{(A_{1},\alpha_{1})}\otimes\cdots\otimes g_{(A_{q},\alpha_{q})}$,
where $\alpha_{1}=1,\ldots,t_{1};\ldots;\alpha_{q}=1,\ldots,t_{q}$. \\

Suppose there are scalars $c_{\alpha_{1},\ldots,\alpha_{q}}\in F$ such
that
\[
0=\sum_{\alpha_{1},\ldots,\alpha_{q}}c_{\alpha_{1},\ldots,\alpha_{q}}\phi_{\mathbf{p}}(g_{(A_{1},\alpha_1)}\otimes\cdots\otimes g_{(A_{q},\alpha_q)})=\sum_{\alpha_{1},\ldots,\alpha_{q}}c_{\alpha_{1},\ldots,\alpha_{q}}\overline{f_{1}\left(g_{(A_{1},\alpha_{1})},\ldots,g_{(A_{q},\alpha_{q})}\right)}.
\]
For $i=1,\ldots,q$, let $\mathbf{a}_{1}^{(A_{i})},\ldots ,\mathbf{a}_{t_{i}}^{(A_{i})}\in A_{i}^{\otimes n_{i}}$
be a dual basis of $\ensuremath{\psi_{\tilde{j}_{i},j_{i}}\circ g_{(A_{i},1)},\ldots,\psi_{\tilde{j}_{i},j_{i}}\circ g_{(A_{i},t_{i})}\in\left(A_{i}^{\otimes n_{i}}\right)^{\ast}}$. Write $\mathbf{a}_{1}^{(A_{1})}=\sum_{l}a_{1,l}\otimes\cdots\otimes a_{n_{1},l}$.
Recall the action of $g_{(A_{1},\alpha_{1})}$ on $\mathbf{a}_{1}^{(A_{i})}$ is determined linearly via the action on $a_{1,l}\otimes\cdots\otimes a_{n_{1},l}$ and the latter is determined via the substitutions
$x_{1}\to a_{1}(l),\ldots,x_{n_{1}}\to a_{n_{1}}(l)$,
where (without loss of generality) $X[A_{1}]=\{x_{1},\ldots,x_{n_{1}}\}$. We remind the reader that $\overline{f_{1}\left(g_{(A_{1},\alpha_{1})},\ldots,g_{(A_{q},\alpha_{q})}\right)} \in \bar{P}_{X_{n};BYE}(A) = P_{X_{n};BYE}^{\tilde{\mathbf{j}},\mathbf{j}}(A)$ and $\tilde{\mathbf{j}},\mathbf{j}$ were chosen in Lemma \ref{lem:HomoToFunctionals}.
Hence we have,

\begin{eqnarray*}
0& = & \sum_{\alpha_{1},\ldots,\alpha_{q}}c_{\alpha_{1},\ldots,\alpha_{q}}\overline{f_{1}\left(g_{(A_1, \alpha_{1})}(\mathbf{a}_{1}^{(A_{1})}),g_{(A_2, \alpha_{2})},\ldots,g_{(A_q, \alpha_{q})}\right)}\\
 & = & \sum_{\alpha_{1}=1,\alpha_{2},\ldots,\alpha_{q}}c_{\alpha_{1}=1,\alpha_{2},\ldots,\alpha_{q}}\overline{f_{1}\left(e_{\tilde{j}_{1},j_{1}}(A_{1}),g_{(A_2, \alpha_{2})},\ldots,g_{(A_q, \alpha_{q})}\right)}.
\end{eqnarray*}
 By considering \textbf{$\mathbf{a}_{1}^{(A_{2})},\ldots,\mathbf{a}_{1}^{(A_{q})}$
}and applying the same argument on the last expression we conclude that
\[
c_{1,\ldots,1}\cdot\overline{f_{1}\left(e_{\tilde{j}_{1},j_{1}}(A_{1}),\ldots,e_{\tilde{j}_{q},j_{q}}(A_{q})\right)}=0.
\]
By Lemma \ref{lem:Consider-a-nonzero}, it follows that $\overline{f_{1}\left(e_{\tilde{j}_{1},j_{1}}(A_{1}),\ldots,e_{\tilde{j}_{q},j_{q}}(A_{q})\right)}\neq0$,
hence $c_{1,\ldots,1}=0$.\\

It is clear that the same argument will work for every $\alpha_{1},\ldots,\alpha_{q}$,
thus every $c_{\alpha_{1},\ldots,\alpha_{q}}=0$.
\end{proof}
Next we study the connection between the different $M(\mathbf{p})$.
\begin{lemma}
\label{lem:MTot1}Fix some $\mathbf{p}_{0}=(X[A_{1}],\ldots,X[A_{q}])$ and denote $\mathbf{\overrightarrow{n}}=(n_{1},\ldots,n_{q})$ where $n_i$ is the number of elements in $X[A_{i}]$. Let $\mathcal{T}=\{e=\tau_{1},\tau_{2},\ldots,\tau_{l}\}$
be a transversal of $S_{\mathbf{p}_{0}}$ in $S_{n}$. Then, the sum of vector spaces
\[
\MM(\mathbf{\overrightarrow{n}}):=M(\tau_{1}\mathbf{p}_{0})+\cdots+M(\tau_{l}\mathbf{p}_{0}),
\]
 is direct.

Note that $\MM(\mathbf{\overrightarrow{n}}) = \overline{FS_{\mathbf{p}_{0}}\cdot f_{1}(x_1,\ldots,x_{n_1},\ldots,x_{n_1+\cdots +n_{q-1}+1},\ldots,x_{n_{q}})}.$
\end{lemma}

\begin{proof}
Suppose
\[
0=\sum_{k=1}^{l}\alpha_{k}h_{k},
\]
where $0\neq h_{k}\in M(\tau_{k}\mathbf{p}_{0})$ and $\alpha_{k}\in F$. Here, $l=\frac{n!}{n_{1}!\cdots n_{q}!}={n \choose n_{1},\ldots,n_{q}}.$

Since

$$h_{1}=\overline{\sum_{\sigma\in S_{p_0}} \beta_{\sigma} f_{1}(x_{\sigma(1)}\cdots x_{\sigma(n_1)},\ldots,x_{\sigma{(n_1+\cdots +n_{q-1}+1)}}\cdots x_{\sigma(n)})}\neq 0$$

we obtain by Lemma \Lemref{Consider-a-nonzero} a nonzero evaluation
$\xi:F\left\{ X_{n},BYE\right\} \to A$ which sends the variables
of $X[A_{i}]$ to elements of $A_{i}$ (for $i=1,\ldots,q$). We claim this evaluation
sends each $h_{k}$ ($k\neq1$) to zero. Indeed, at least one of the
sets $\tau_{k}X[A_{1}],\ldots,\tau_{k}X[A_{q}]$ must have an element
$x\in\tau_{i_{0}}X[A_{i_{0}}]$ which is sent to some $A_{i}$, $i \neq i_0$ and hence

$$h_{k}=\overline{\sum_{\sigma\in S_{p_0}} \beta_{\tau_{k}\sigma} f_{1}(x_{\tau_{k}\sigma(1)}\cdots x_{\tau_{k}\sigma(n_1)},\ldots,x_{\tau_{k}\sigma{(n_1+\cdots +n_{q-1}+1)}}\cdots x_{\tau_{k}\sigma(n)})}$$

is sent to zero. It follows that $0=\alpha_{1}\cdot\xi(h_{1})$ and hence $\alpha_{1}=0$.

Repeating this argument for $h_{2},\ldots,h_{l}$ yields $\alpha_{2}=\alpha_{3}= \cdots =\alpha_{l}=0$.
\end{proof}
Write
\[
\MM(n):=\sum_{n_{1}+\cdots+n_{q}=n}\MM\left(\overrightarrow{n}=(n_{1},\ldots,n_{q})\right).
\]
Note that in view of our notation above for $\MM(\overrightarrow{n})$ we have
$$\MM(n) = \overline{FS_{n}\cdot f_{1}(x_1,\ldots,x_{n_1},\ldots,x_{n_1+\cdots +n_{q-1}+1},\ldots,x_{n_{q}})}.$$

Using a similar argument as in the previous lemma we obtain the following (proof is omitted).

\begin{lemma}\label{lem:MTot2}
For every $n$, the above sum in $\MM(n)$ is direct.
\end{lemma}

Consider now the group $S_{X_{n}\cup\{w_{1},\ldots,w_{s}\}}$ (see the
beginning of the section for the definition of $w_{1},\ldots,w_{s}$)
and let $\{e=\sigma_{1},\sigma_{2},\ldots,\sigma_{u}\}$ be a transversal
of $S_{X_{n}}$ inside the aforementioned group. Note that $u=\frac{(n+s)!}{n!}=\Theta(n^{s}).$

\begin{lemma} \label{including the radical variables}
\label{lem:MTOT3}The sum
\[
\MMM(n)=\sum_{k=1}^{u}\sigma_{k}\MM(n)
\]
 is direct, where
\[
\sigma_{k}\MM(n)=\overline{\sigma_{k}FS_{n}\cdot f_{1}(x_1,\ldots,x_{n_1},\ldots,x_{n_1+\cdots +n_{q-1}+1},\ldots,x_{n_{q}};w_1,\ldots,w_s)}.
\]
\end{lemma}

\begin{proof}
The idea used in the proof of \lemref{MTot1} works also here. Suppose
\[
0=\sum_{k=1}^{u}\alpha_{k}h_{k},
\]
where $0\neq h_{k}\in\sigma_{k}\MM(n)$ and $\alpha_{k}\in F$. Since
$h_{1}\neq0$, we obtain by \lemref{Consider-a-nonzero},
a nonzero evaluation $\xi:F\left\{ X_{n},BYE\right\} \to A$ which
sends the variables in $X_{n}$ to elements of $A_{ss}$ and $w_{1},\ldots,w_{q}$
to elements of $J$. This evaluation sends each $h_{k}$ ($k\neq1$)
to zero. Indeed, $\sigma_{k}B_{1},\ldots,\sigma_{k}B_{s}$ are alternating
sets of cardinality $d+1$ in $h_{k}$. However, there is some $i_{0}$
for which $\sigma_{k}B_{i_{0}}$ does not contain any of the variables
$w_{1},\ldots,w_{s}$. Hence $\xi(\sigma_{k}B_{i_{0}})\subseteq A_{ss}$.
As a result, $\xi(h_{k})=0$. Thus, $0=\alpha_{1}\xi(h_{1})$ and
so $\alpha_{1}=0$.

Repeating this argument to $h_{2},\ldots,h_{l}$ yields $\alpha_{2}=\cdots =\alpha_{l}=0$.
\end{proof}
Now we show that asymptotically $\dim_{F}\MM(n)$ has the desired
lower bound. Recall that $c_{n_{i}}(A_{i})\sim K_{i}n^{\frac{1-d_{i}^{2}}{2}}d_{i}^{2n}$
for some constant $K_{i}\in\mathbb{R}$ {\cite{RegAsymp}}.
As a result of \lemref{MTot1}, \lemref{MTot2} and the above lemma we get (the inequality $\lesssim$ is asymptotically)
\[
\sum\limits _{n_{1}+\ldots+n_{q}=n}\frac{K_{1}\cdots K_{q}}{d_{1}^{2}\cdots d_{q}^{2}}\cdot{n \choose n_{1},\ldots,n_{q}}\cdot n_{1}^{\frac{1-d_{1}^{2}}{2}}\cdots n_{q}^{\frac{1-d_{q}^{2}}{2}}\cdot d_{1}^{2n_{1}}\cdots d_{q}^{2n_{q}}\cdot n^{s}\lesssim\dim_{F}\MMM(n).
\]

Finally we apply Theorem \ref{RegevBeckner} and obtain
\[
\dim_{F}\MMM(n)\geq Cn^{\frac{q-d}{2}+s}d^{n}
\]
 (here $C$ is some nonnegative constant).
\begin{corollary}[Lower bound]
 Let $A=A_{1}\times\cdots\times A_{q}\oplus J(A)$ be a basic algebra
of Kemer index $\kappa_{A}=(d,s)$. Then
\[
c_{n}(A)\geq Cn^{\frac{q-d}{2}+s}d^{n}.
\]
 for some constant $C\in\mathbb{R}$.\end{corollary}
\begin{proof}
It is enough to show that
\[
\dim_{F}c_{n+\gamma}(A)\geq\dim_{F}\MMM(n),
\]
 for some $\gamma$ independent of $n$. Indeed,
\[
\MMM(n)\subseteq\overline{P_{X_{n};BYE}(A)}
\]
and $\overline{P_{X_{n};BYE}(A)}$ is a projection of $P_{X_{n};BYE}(A)$.
Hence, the statement stands for $\gamma=|BYE|$.
\end{proof}
So combined with Theorem \ref{Upper bound result} we
finally get a positive answer on Giambruno's conjecture.
\begin{theorem}\label{main theorem}
Let $A=A_{1}\times\cdots\times A_{q}\oplus J(A)$ be a basic algebra
of Kemer index $\kappa_{A}=(d,s)$. Then
\[
c_{n}(A)=\Theta(n^{\frac{q-d}{2}+s}d^{n}).
\]
In the special case where the algebra $A$ has a unit we have
\[
c_{n}(A)\simeq Cn^{\frac{q-d}{2}+s}d^{n}.
\]
for some constant $0<C\in\mathbb{R}$.
\end{theorem}

\end{document}